\documentclass{amsart}

\usepackage{amssymb}
\usepackage{picture}

\newtheorem{theorem}{Theorem}[section]
\newtheorem{lemma}[theorem]{Lemma}
\newtheorem{example}[theorem]{Example}
\newtheorem{proposition}[theorem]{Proposition}
\newtheorem{problem}[theorem]{Problem}
\newtheorem{corollary}[theorem]{Corollary}

\newcommand{\N}{\mathbb N}

\newcommand{\R}{\mathbb R}

\newcommand{\on}{\operatorname}
\author{Szymon G\l \c ab}
\address{Institute of Mathematics, \L \'od\'z University of Technology,
W\'olcza\'nska 215, 93-005 \L \'od\'z, Poland}
\email {szymon.glab@p.lodz.pl}

\author{Jacek Marchwicki}
\address{Institute of Mathematics, \L \'od\'z University of Technology,
W\'olcza\'nska 215, 93-005 \L \'od\'z, Poland}
\email {marchewajaclaw@gmail.com}

\title[Levy-Steinitz theorem and achievement sets]{Levy-Steinitz theorem and achievement sets of conditionally convergent series on the real plane}

\thanks{The second author has been supported by the National Science Centre Poland Grant no. DEC-2012/07/D/ST1/02087.}
\subjclass[2010]{Primary: 40A05 ; Secondary: 11K31} 
\keywords{achievement set, set of subsums, conditionally convergent series, sum range}

\begin{document}

\begin{abstract}
Levy-Steinitz theorem characterize sum range of conditionally convergent series, that is a set of all its convergent rearrangements; in finitely dimensional spaces -- it is an affine subspace. An achievement of a series is a set of all its subsums. We study the properties of achievement sets of series whose sum range is the whole plane. It turns out that it varies on the number of Levy vectors of a series.        
\end{abstract} 

\maketitle

\section{Introduction}

The old result due to Riemann states that for any $a\in\R$ and for any conditionally convergent series $\sum_{n=1}^\infty x_n$ of real numbers there is a permutation $\sigma$ of natural numbers such that the rearrangement $\sum_{n=1}^\infty x_{\sigma(n)}$ is convergent to $a$. This fact can be generalized to finitely dimensional spaces as follows. By $\on{SR}(x_n)$ denote the sum range of the series, that is the set of sums of all convergent rearrangements of the series $\sum_{n=1}^\infty x_n$. If $\sum_{n=1}^\infty x_n$ is conditionally convergent in $\R^m$, then the classical Levy-Steinitz Theorem states that $\on{SR}(x_n)$ is an affine subset of the underlying space. More precisely, $\on{SR}(x_n)=\sum_{n=1}^\infty x_n+\Gamma^{\perp}$ where $\Gamma^{\perp}$ is a subspace orthogonal to the set $\Gamma=\{f\in(\R^m)^*:\sum_{n=1}^\infty\vert f(x_n)\vert<\infty\}$ of all functionals of series convergence. The theory of rearrangements of conditionally convergent series in Banach spaces, and further in topological vector spaces, has been developed and deeply investigated by many authors; we refer the reader to the monograph \cite{KadKad} for details. 

An important tool to study sum range $\on{SR}(x_n)$ is a so-called Levy vector. A vector $u\in\mathbb{R}^2$, $\Vert u\Vert =1$ is called the Levy vector of a series $\sum_{n=1}^{\infty}v_n$ if for every $\varepsilon>0$ we have $\sum_{v_n\in S_{\varepsilon}(u)} \Vert v_{n}\Vert=\infty$, where  $S_{\varepsilon}(u)=\{v : \langle u,v\rangle \geq (1-\varepsilon) \Vert u\Vert \Vert v \Vert\}$. Absolutely convergent series has no Levy vectors. If $\sum_{n=1}^\infty x_n$, $x_n\in \R^m$, is conditionally convergent, then the set $\on{L}(x_n)$ of all its Levy vectors is non-empty and each closed half-sphere contains a Levy vector of $\sum_{n=1}^\infty x_n$. Moreover, the linear subspace $-\sum_{n=1}^\infty x_n+\on{SR}(x_n)$ of $\R^m$ contains a linear space spanned by $\on{L}(x_n)$. However in general $\on{span}(\on{L}(x_n))$ not need to be equal to $-\sum_{n=1}^\infty x_n+\on{SR}(x_n)$; for example if $x_n=(\frac{(-1)^n}{n},\frac{(-1)^n}{\sqrt{n}})$, then $\on{L}(x_n)=\{(0,1),(0,-1)\}$ but $\on{SR}(x_n)=\R^2$. From the sum range point of view there are only two types of conditionally convergent series $\sum_{n=1}^\infty x_n$ on the plane $\R^2$ -- those for which $\on{SR}(x_n)$ is a line and those for which $\on{SR}(x_n)=\R^2$. There is the other set which can be naturally attained to a series and which can distinguish series within these two classes -- the achievement set. 

For a sequence $(x_{n})$ in Banach space, we call the set $\on{A}(x_{n})=\{\sum_{n=1}^\infty\varepsilon_nx_n : (\varepsilon_{n})\in\{0,1\}^{\mathbb{N}}\}$ \emph{the set of subsums} or \emph{the achievement set}. This notion was mostly studied for absolutely summable sequences on the real line. Probably the first paper where topological properties of the achievement sets were investigated is that of Kakeya \cite{Kakeya}. He proved that such sets can be finite sets, finite unions of compact intervals or homeomorphic to the Cantor set. His conjecture that the achievement set can be only one of these three forms was disproved by Weinstein and Shapiro \cite{WS}, Ferens \cite{Ferens} and Guthrie and Nymann \cite{GN}. It is worth to mention that the motivation of Ferens' paper \cite{Ferens} came from measure theory; namely, the Author construct purely atomic probabilistic measure the range of which is neither finite union of intervals nor homeomorphic to the Cantor sets. Topological classification of achievement sets on the real line was given by Guthrie, Nymann and Saenz \cite{GN,NS1} who proved that they can be a finite set, a finite union of intervals, homeomorphic to the Cantor set or it can be a so called Cantorval (see also \cite{BFPW}). A Cantorval is a set homeomorphic to the union of the Cantor set and sets which are removed from the unite segment by even steps of the Cantor set construction. If underlying sequence is regular, for example multigeometric, then achievement sets are fractals. Fractal geometry of achievement sets were studied in \cite{BBFS,BG,BFS,MO,M1,M2}. 

The achievement sets of conditionally convergent series of real numbers is the whole real line \cite{BFPW,Jones,N1,N2}. Studies of achievement sets of conditionally convergent series in multidimensional spaces were initiated by us in \cite{BGM}. 
In that paper we focused mostly on the case when $\sum_{n=1}^\infty x_n$ is conditionally convergent on $\R^2$ and $\on{SR}(x_n)$ is a line. We made also a general observation that $\on{SR}(x_n)=\R^m$ if and only if the closure of $\on{A}(x_{n})$ equals $\R^m$ as well. There is also an example of series on the plane such that $\on{SR}(x_n)=\R^2$ and $\on{A}(x_{n})$ is dense and null. This paper is devoted to answer the following question. What need to be assumed on the series $\sum_{n=1}^\infty x_n$ with $\on{SR}(x_n)=\R^2$ to obtain $\on{A}(x_{n})=\R^2$? The answer depends firstly on the number of Levy vectors. 
In Section \ref{SectionMoreThan2Levy} we show that if a series has more than two Levy vectors, then $\on{A}(x_n)=\on{SR}(x_n)=\R^2$. We are able to prove even more: for any $a\in\R^2$ there is an increasing sequence $(n_k)$ of indexes such that $\sum_{k=1}^\infty x_{n_k}$ is absolutely convergent to $a$. In symbols $\on{A}_{\on{abs}}(x_{n})=\R^2$ where $\on{A}_{\on{abs}}(x_{n})=\{\sum_{n=1}^{\infty} \varepsilon_{n} x_{n}  : \sum_{n=1}^{\infty} \varepsilon_{n}\Vert x_{n}  \Vert<\infty,  \varepsilon_{n}\in\{0,1\} \ \text{for each} \  n\in\mathbb{N} \}$. Note that always $\on{A}_{\on{abs}}(x_{n})\subset\on{A}(x_n)$. In Section \ref{Section2Levy} we discuss the case when there are exactly two Levy vectors. We find a necessary condition for $\on{A}(x_n)=\R^2$ for series $\sum_{n=1}^\infty x_n$ with $\vert \on{L}(x_n)\vert=2$. We give an example of series $\sum_{n=1}^\infty x_n$ with two Levy vectors and such that $\on{A}_{\on{abs}}(x_n)\neq \on{A}(x_n)=\R^2$. Finally in Section \ref{SectionExample} we give an example of series such that $\on{SR}(x_n)=\R^2$, $\vert \on{L}(x_n)\vert=2$ and $\on{A}(x_n)$ is a graph of a partial function, more precisely all its vertical sections have got at most one element. 

The idea of using Levy vectors to analyze conditionally convergent series come to us from Klinga's paper \cite{Klinga} who in turn used the ideas from Halperin's paper \cite{H}. Klinga used Levy vectors to study ideal version of Levy-Steinitz theorem and in the proof of the main result of \cite{Klinga} he distinguish between situation when a series has two Levy vectors or more than two. We found this distinction crucial in our consideration.

\section{Series with more than two Levy vectors}\label{SectionMoreThan2Levy}

We start this section with a simple observation.

\begin{lemma}\label{LemmaApprox}
Let $(a_n)$ be a sequence of positive real numbers such that $a_n\to 0$ and $\sum_{n=1}^\infty a_n=\infty$. Let $x,\varepsilon>0$ and $N\in\N$. There exists finite set of indexes $F$ such that $x-\varepsilon<\sum_{n\in F}a_n<x$ and $N<n$ for every $n\in F$.  
\end{lemma}

The notion of Levy vectors was investigated by Halperin in \cite{H} who proved the following auxiliary result.

\begin{proposition}\label{minimumdwa} Let $\sum_{n=1}^{\infty}v_n$ be a conditionally convergent series on the plane. \begin{itemize}
\item[(i)] If $\on{SR}(v_n)=\mathbb{R}^2$, then in every closed half unite circle there is a Levy vector.
\item[(ii)] $\on{L}(v_n)$ is a closed subset of the unit circle. 
\item[(iii)] $u$ is a Levy vector of a series $\sum_{n=1}^{\infty}v_n$ if and only if there exists a subsequence $(k_n)$ such that $v_{k_n}=\alpha_{n}u+w_{n}$, where $(\alpha_{n})$ is a sequence of positive real numbers tending to zero such that $\sum_{n=1}^{\infty}\alpha_{n}=\infty$ and $\sum_{n=1}^{\infty}\Vert w_{n}\Vert<\infty$.
\end{itemize}
\end{proposition}

Let us give a couple of examples to illustrate the notion of Levy vectors. 
\begin{itemize}
\item The series $\sum_{n=1}^\infty(\frac{(-1)^n}{n},0)$ and $\sum_{n=1}^\infty(\frac{(-1)^n}{n},\frac{(-1)^n}{n})$ have one-dimensional sum ranges and therefore each of them has exactly two Levy vectors. 
\item The sum range of the series $\sum_{n=1}^\infty(\frac{(-1)^n}{n},\frac{(-1)^n}{\sqrt{n}})$ is the whole plane, while it has only two Levy vectors.
\item Let $z\in\mathbb{C}$ be a complex number such that $\vert z\vert=1$, which is not a root of unity, that is $z^n\neq 1$ for every $n\in\mathbb{N}$. Then $\sum_{n=1}^{\infty}\frac{z^n}{n}$ is a conditionally convergent series. Moreover the set of all Levy vectors of  $\sum_{n=1}^{\infty}\frac{z^n}{n}$ is the whole unit circle.
\item Let $x_{2n-1}=(0,\frac{(-1)^n}{n})$, $x_{2n}=(\frac{(-1)^n}{n},0)$ for every $n\in\mathbb{N}$. Then $\on{L}(x_n)=\{(0,-1),(-1,0),(0,1),(1,0)\}$.
\item Let $x_{3n-2}=(-\frac{1}{3n-2},-\frac{1}{3n-2})$, $x_{3n-1}=(\frac{1}{3n-1},0)$, $x_{3n}=(0,\frac{1}{3n})$ for every $n\in\mathbb{N}$. Then $\on{L}(x_n)=\{(-1,-1),(1,0),(0,1)\}$.
\item Let $x_{3n-2}=(-\frac{1}{3n-2},-\frac{1}{\sqrt{3n-2}})$, $x_{3n-1}=(\frac{1}{3n-1},0)$, $x_{3n}=(0,\frac{1}{\sqrt{3n}})$ for every $n\in\mathbb{N}$. Then $\on{L}(x_n)=\{(0,-1),(1,0),(0,1)\}$. Note that $\on{L}(x_n)$ is disjoint with open half space $\{(x,y):x<0\}$.  
\end{itemize}

We say that $H$ is a \emph{central open half space} if $H$ is a rotation of $\{(x,y):x>0\}$. The last example shows that the set $\on{L}(x_n)$ of Levy vectors can be disjoint with some central open half space even when the sum range of the underlying series is the whole plane.

Let $v_1,v_2$ be vectors on the plane. Put $\on{span^{+}}(v_1,v_2):=\{av_1+bv_2 : a,b>0\}$. If $v_1,v_2$ are colinear, then $\on{span^{+}}(v_1,v_2)$ is the line or halfline;  If $v_1,v_2$ are non-colinear, then $\on{span^{+}}(v_1,v_2)$ is the interior of the cone consisting of the points laying between half-lines $\{cv_1 : c>0\}$ and  $\{cv_2 : c>0\}$. From Proposition \ref{minimumdwa} we immediately obtain the following. 
\begin{corollary} \label{centralrozpiety}
Assume that a series  $\sum_{n=1}^{\infty}v_n$ has at least three Levy vectors. Then there exists a central open half space $H$ and $u,v,w\in \on{L}(v_n)$ such that $H\subset \on{span^{+}}(u,v)\cup \on{span^{+}}(v,w)\cup L$, where $L$ is a halfline $\{av: a>0\}$.
\end{corollary}

\begin{lemma} \label{przyblizanie}
Let $\sum_{n=1}^{\infty}x_n$ be a conditionally convergent series with Levy vectors $u,v$ and $x\in span^{+}(u,v)$. 
There are $\delta>0$, $(c_n)$ of positive numbers and a sequence of finite, pairwise disjoint sets of indexes $(P_n)$ such that
\begin{itemize}
\item[(i)] $c_n<\frac{\delta}{2^{n}}$ and $\frac{1}{2}c_n<c_{n+1}<c_{n}$;
\item[(ii)] $\min P_{n+1}>\max P_n$;
\item[(iii)] $\Vert\sum_{k=1}^{n} y_k - (1-\frac{1}{2^{n}})x\Vert<c_n$ where $y_k=\sum_{i\in P_k}x_{i}$;
\item[(iv)] $\sum_{k\in P_{n+1}}\Vert x_k\Vert\leq \frac{1}{2}(c_{n+1}-\frac{c_n}{2})+\frac{1}{2}(a_n+b_n)$, where $a_n,b_n$ are such that $x-\sum_{k=1}^{n} y_k=a_nu+b_nv$.
\end{itemize}
\end{lemma}

\begin{proof}
Fix $x\in\on{span^{+}}(u,v)$, $x=a_{0}u+b_{0}v$ for $a_{0},b_{0}>0$. Let $\delta=\min\{d(x,U),d(x,V)\}$, where $U=\{cu: c>0\}$, $V=\{cv : c>0\}$ and $d(x,X)$ is a distance between a point $x$ and a set $X$. We define $(c_n)$ by simple induction. Find $\frac{\delta}{4}<c_1<\frac{\delta}{2}$. Then take $\frac{c_1}{2}<c_2<\frac{\delta}{4}$ and so on. Let $x_{k_n}=\alpha_{k_n}^{u}u+w_{k_n}^{u}$ and  $x_{m_n}=\alpha_{m_n}^{v}v+w_{m_n}^{v}$ for $n\in\mathbb{N}$ satisfy the assertion of Proposition \ref{minimumdwa}(iii). Moreover assume that $(k_n)$ and $(m_n)$ are disjoint. Using Lemma \ref{LemmaApprox} we find finite sets of indexes $A_1\subset\{k_n:n\in\N\}$ and $B_1\subset\{m_n:n\in\N\}$ such that $\frac{a_{0}}{2}-\frac{c_1}{4}<\sum_{n\in A_1} \alpha_{n}^{u}<\frac{a_{0}}{2}$, $\sum_{n\in A_1}\Vert w^u_n\Vert<\frac{c_1}{4}$, $\frac{b_{0}}{2}-\frac{c_1}{4}<\sum_{n\in B_1} \alpha_{n}^{v}<\frac{b_{0}}{2}$ and $\sum_{n\in B_1}\Vert w^v_n\Vert<\frac{c_1}{4}$.
Observe that
\[\Vert \sum_{n\in A_1} x_n +  \sum_{n\in B_1}x_n-\frac{x}{2}\Vert=
\Vert \sum_{n\in A_1} x_n +  \sum_{n\in B_1} x_n-\frac{a_{0}}{2}u-\frac{b_{0}}{2}v\Vert=
\Vert \sum_{n\in A_1} (\alpha_n^{u}u +w_{n}^{u}) + \sum_{n\in B_1} (\alpha_{n}^{v}v +w_{n}^{v})-\frac{a_{0}}{2}u-\frac{b_{0}}{2}v\Vert=
\]
\[
\Vert \sum_{n\in A_1} \alpha_{n}^{u}u-\frac{a_{0}}{2}u +\sum_{n\in A_1}w_{n}^{u} + \sum_{n\in B_1} \alpha_{n}^{v}v-\frac{b_{0}}{2}v +\sum_{n\in B_1}w_{n}^{v}\Vert\leq  
\]
\[
\vert \sum_{n\in A_1} \alpha_{n}^{u}-\frac{a_{0}}{2}\vert \Vert u\Vert + \sum_{n\in A_1}\Vert w_{n}^{u}\Vert + \vert\sum_{n\in B_1} \alpha_{n}^{v}-\frac{b_{0}}{2}\vert \Vert v \Vert+\Vert\sum_{n\in B_1}w_{n}^{v}\Vert<
4\cdot \frac{c_1}{4}=c_1.\]
Put $P_{1}:=A_1\cup B_1$. Thus $\Vert y_{1}-\frac{x}{2}\Vert<c_{1}$.

Assume that we have constructed $(P_k)_{k=1}^{t}$ for some $t\in\mathbb{N}$. We know that $\Vert \sum_{k=1}^{t} y_k - (1-\frac{1}{2^{t}})x\Vert<c_t$, so $\Vert(x- \sum_{k=1}^{t} y_k)-\frac{1}{2^{t}}x \Vert<c_t<\frac{\delta}{2^t}$. It easy to see that $\min\{d(\frac{x}{2^t},U),d(\frac{x}{2^{t}},V)\}=\frac{\delta}{2^t}$, so $x- \sum_{k=1}^{t} y_k\in B(\frac{1}{2^{t}}x,\frac{\delta}{2^t})$, which implies $x- \sum_{k=1}^{t} y_k\in \on{span}^{+}(u,v)$. Let $x- \sum_{k=1}^{t} y_k=a_t u+b_t v$. Let $A_{t+1}\subset\{k_n>\max P_t:n\in\N\}$ and $B_{t+1}\subset\{m_n>\max P_t:n\in\N\}$ be finite sets such that $\sum_{n\in A_{t+1}}\Vert w_{n}^{u}\Vert<\frac{\varepsilon_t}{4}$, $\sum_{n\in B_{t+1}}\Vert w_{n}^{v}\Vert<\frac{\varepsilon_t}{4}$, $\frac{a_{t}}{2}-\frac{\varepsilon_t}{4}<\sum_{n\in A_{t+1}} \alpha_{n}^{u}<\frac{a_{t}}{2}$ and $\frac{b_{t}}{2}-\frac{\varepsilon_t}{4}<\sum_{n\in B_{t+1}} \alpha_{n}^{v}<\frac{b_{t}}{2}$ where $\varepsilon_t:=c_{t+1}-\frac{1}{2}c_{t}$. Put $P_{t+1}=A_{t+1}\cup B_{t+1}$. Thus
$$
\Vert \sum_{k=1}^{t} y_k+2y_{t+1}-x  \Vert=
\Vert x- \sum_{k=1}^{t} y_k-2(\sum_{n\in A_{t+1}} x_{n}+\sum_{n\in B_{t+1}}x_{n})\Vert=$$
$$\Vert a_{t}u+b_{t}v -2\sum_{n\in A_{t+1}} \alpha_{n}^{u}u-2\sum_{n\in A_{t+1}}w_{n}^{u}-2\sum_{n\in B_{t+1}}\alpha_{n}^{v}v-2\sum_{n\in B_{t+1}}w_{n}^{v}\Vert\leq $$
$$2\vert \frac{a_t}{2}-\sum_{n\in A_{t+1}} \alpha_{n}^{u}\vert \Vert u\Vert + 2\sum_{n\in A_{t+1}}\Vert w_{n}^{u}\Vert+2\vert \frac{b_t}{2}-\sum_{n\in B_{t+1}} \alpha_{n}^{v}\vert \Vert v\Vert + 2\sum_{n\in B_{t+1}}\Vert w_{n}^{v}\Vert\leq 2\cdot 4\cdot \frac{\varepsilon_{t}}{4}=2\varepsilon_{t}.
$$
Hence
\[
\Vert \sum_{k=1}^{t+1} y_k - (1-\frac{1}{2^{t+1}})x\Vert\leq\frac{1}{2}\Vert\sum_{k=1}^{t} y_k - (1-\frac{1}{2^{t}})x\Vert+\frac{1}{2}\Vert \sum_{k=1}^{t} y_k+2y_{t+1}-x  \Vert<\frac{1}{2}c_{t}+\varepsilon_{t}=c_{t+1}.
\]
Moreover for any $t\in\N$ we have
\[
\sum_{l\in P_{t+1}}\Vert x_{l}\Vert=
\sum_{n\in A_{t+1}}\Vert  \alpha_{n}^{u}u +w_{n}^{u} \Vert+\sum_{n\in B_{t+1}}\Vert \alpha_{n}^{v}v +w_{n}^{v}\Vert\leq 
\sum_{n\in A_{t+1}} (\vert\alpha_{n}^{u}\vert+\Vert w_{n}^{u} \Vert)+\sum_{n\in B_{t+1}}(\vert \alpha_{n}^{v}\vert +\Vert w_{n}^{v}\Vert)\leq
\]
\[
\frac{a_{t}}{{2}}+\frac{\varepsilon_{t}}{4}+\frac{b_{t}}{{2}} +\frac{\varepsilon_{t}}{4}=\frac{1}{2}\varepsilon_{t}+\frac{1}{2}a_t +\frac{1}{2}b_t.
\]

\end{proof}

\begin{theorem}\label{span}
Let $\sum_{n=1}^{\infty}x_n$ be a conditionally convergent series with two non-colinear Levy vectors $u,v$. Then $\on{span^{+}}(u,v)\subset \on{A}_{\on{abs}}(x_n)$.
\end{theorem}

\begin{proof}
Let $x\in \on{span^{+}}(u,v)$. Let $(P_t)$ be the sequence of sets of indexes from the assertion of Lemma \ref{przyblizanie}.  We need to show $\sum_{l\in\bigcup_{t=1}^{\infty}P_{t}} \Vert x_{l}\Vert$ is convergent. Note that 
\[
\sum_{t=1}^{\infty}(c_{t+1}-\frac{1}{2}c_t)\leq\sum_{t=1}^{\infty}(c_{t}-\frac{1}{2}c_{t})\leq \frac{1}{2}\sum_{t=1}^{\infty}\frac{\delta}{2^t}=\frac{\delta}{2}<\infty
\]
Moreover, if $\alpha\in(0,\pi)$ is such that $\cos(\alpha)=\vert\langle u,v\rangle\vert$, then for each $t\in\mathbb{N}$ we have
$\Vert a_{t}u+b_{t}v\Vert\geq \Vert a_{t}u-a_t \cos(\alpha)v\Vert\geq a_t\Vert u-\cos(\alpha)v\Vert$. Therefore by Lemma  \ref{przyblizanie}(iii) we obtain
\[
a_t\leq\frac{\Vert a_{t}u+b_{t}v\Vert}{\Vert u-\cos(\alpha)v\Vert}\leq \frac{c_{t}+\frac{1}{2^t}\Vert x\Vert}{\Vert u-\cos(\alpha)v\Vert}
\]
and consequently 
\[
\sum_{t=1}^{\infty}a_{t}\leq \frac{1}{\Vert u-\cos(\alpha)v\Vert} \sum_{t=1}^{\infty} (c_{t}+\frac{1}{2^t}\Vert x\Vert)<\frac{\delta+\Vert x\Vert}{\Vert u-\cos(\alpha)v\Vert}<\infty.
\]
In similar way one can show that $\sum_{t=1}^{\infty}b_{t}<\infty$. Thus by Lemma \ref{przyblizanie}(iv)
\[
\sum_{l\in\bigcup_{t=1}^{\infty}P_{t+1}} \Vert x_{l}\Vert=\sum_{t=1}^{\infty}\sum_{l\in P_{t+1}}\Vert x_{l}\Vert\leq\frac{1}{2}\sum_{t=1}^{\infty}a_t +\frac{1}{2}\sum_{t=1}^{\infty}b_t +\frac{1}{2}\sum_{t=1}^{\infty}(c_{t+1}-\frac{1}{2} c_t)<\infty.
\]
Since $P_1$ is a finite set we have $\sum_{l\in\bigcup_{t=1}^{\infty}P_{t}} \Vert x_{l}\Vert<\infty$. Hence $\sum_{l\in\bigcup_{t=1}^{\infty}P_{t}}x_{l}$ is convergent. By Lemma \ref{przyblizanie}(i) we know that $c_n\to 0$, so by Lemma \ref{przyblizanie}(iii) we obtain that $x=\sum_{l\in\bigcup_{t=1}^{\infty}P_{t}}x_{l}$. 
\end{proof}

We show that there is a strict connection between geometry of the set of Levy vectors and geometry of $\on{A}_{\on{abs}}(x_n)$. 

\begin{theorem}\label{openhalfcircle}
Let $\sum_{n=1}^\infty x_n$ be a conditionally convergent series on the real plane. If any open half circle contains at least one Levy vector of the series, then $\on{A}_{\on{abs}}(x_n)=\R^2$. 
\end{theorem}
\begin{proof}
Since every open half circle contains at least one Levy vector we can find three Levy vectors $u,v,w$ of the series $\sum_{n=1}^\infty x_n$, which satisfies $\on{span}(u)\neq \on{span}(v)\neq \on{span}(w)\neq \on{span}(u)$ and 
$$\on{span^{+}}(u,v)\cup \on{span^{+}}(v,w)\cup\on{span^{+}}(w,u)\cup \on{span^{+}}(u)\cup\on{span^{+}}(v)\cup \on{span^{+}}(w)=\mathbb{R}^2$$ 
where $\on{span^{+}}(z)=\{az: a>0\}$. By Theorem \ref{span} we have $\mathbb{R}^2\setminus(\on{span^{+}}(u)\cup\on{span^{+}}(v)\cup \on{span^{+}}(w))\subset \on{A}_{\on{abs}}(x_n)$. Fix $x\in\on{span^{+}}(u)$. There exists $y\in  \on{span^{+}}(u,v)$ such that $x-y\in \on{span^{+}}(u,w)$. Hence by Theorem \ref{span} one can find $E\subset\mathbb{N}$ such that $y=\sum_{n\in E} x_n$, where $\sum_{n\in E} \Vert x_n\Vert<\infty$. Since an absolutely convergent series does not affect Levy vectors, we have $\on{A}((x_n)_{n\in\mathbb{N}\setminus E})\supset \on{span^{+}}(u,v)\cup \on{span^{+}}(v,w)\cup \on{span^{+}}(w,u)$, especially $x-y=\sum_{n\in F} x_n$ and $\sum_{n\in F} \Vert x_n\Vert<\infty$ for some $F\subset\mathbb{N}\setminus E$. Hence $\sum_{n\in E\cup F}x_n=\sum_{n\in E}x_n+\sum_{n\in F}x_n=y+x-y=x$, so $x\in \on{A}_{\on{abs}}(x_n)$. If $x\in \on{span^{+}}(v)$ or  $x\in \on{span^{+}}(w)$ then the proof is analogous. Hence  $\on{A}_{\on{abs}}(x_n)=\R^2$. 
\end{proof}

\begin{theorem}\label{3LevyVectorsAndABS}
Let $\sum_{n=1}^\infty x_n$ be a conditionally convergent series on the real plane with more than two Levy vectors. Then $\on{A}_{\on{abs}}(x_n)=\R^2$.
\end{theorem}
\begin{proof}
By Theorem \ref{openhalfcircle} it is enough to show that for a series with $\on{L}(x_n)=\{u,-u,v\}$, where $v\notin span(u)$, we have $\on{A}_{\on{abs}}(x_n)=\R^2$. Without losing generality we may assume $u=(0,1)$ and $v\in H=\{(x,y): x>0\}$. By Theorem \ref{span} and by using method of the proof of Theorem \ref{openhalfcircle} we get inclusion $H\subset\on{A}_{\on{abs}}(x_n)$. 
Fix $x\notin H$. Since $\overline{A(x_n)}=\on{SR}(x_n)=\R^2$ one can find finite set $E\subset\mathbb{N}$ such that $x-\sum_{n\in E}x_n\in H$. We have $L((x_n)_{n\in\mathbb{N}\setminus E})=\{u,-u,v\}$, so $\on{A}_{\on{abs}}((x_n)_{n\in\mathbb{N}\setminus E})\supset H$. Let $F\subset\mathbb{N}\setminus E$ be such that $\sum_{n\in F}x_n=x-\sum_{n\in E}x_n$ and $\sum_{n\in F}\Vert x_n\Vert<\infty$. Hence $\sum_{n\in E\cup F}x_n=x$.
\end{proof}


\section{Series with two Levy vectors}\label{Section2Levy}

In this section we give the sufficient condition for series $\sum_{n=1}^\infty v_n$ with two Levy vectors for $\on{A}(v_n)=\R^2$. We call this condition as reduction property. We say that a series $\sum_{n=1}^{\infty} (x_{n},y_{n})$ has a \emph{reduction property} if $\lim_{n\to\infty}\frac{\vert y_n\vert}{\vert x_n\vert}=\infty$ and for every $\varepsilon>0$, $x\in(-\varepsilon,\varepsilon)$, $N\in\mathbb{N}$ there exists a finite set of indices $A=\{k_1<k_2<\ldots<k_m\}$ with $k_1\geq N$ such that $\vert\sum_{n\in A}x_n-x\vert<\frac{\varepsilon}{2}$ and $\max_{1\leq j\leq m} \vert \sum_{n=1}^{j}x_{k_n}\vert<\varepsilon$ and $\max_{1\leq j\leq m} \vert \sum_{n=1}^{j}y_{k_n}\vert<\varepsilon$.

\begin{proposition}\label{plaszczyznasumrange}
Let $\sum_{n=1}^{\infty} (x_{n},y_{n})$ be a series, which has the reduction property. If $\sum_{n=1}^{\infty} (x_{n},y_{n})$ is a conditionally convergent series, then its sum range is the whole plane, in symbols $\on{SR} (x_{n},y_{n})=\mathbb{R}^2$. 
\end{proposition}

\begin{proof}
Observe that $\sum_{n=1}^{\infty}  x_{n}$ can not be absolutely convergent, namely we get $\sum_{n=1}^{\infty} \vert x_{n}\vert=\infty$. Indeed fix $\varepsilon>0$, $x=\frac{3}{4}\varepsilon$. Then one can find a family of pairwise disjoint finite sets of indices $(A_{t})_{t=1}^{\infty}$ with $\min A_{t+1}>\max A_t$ such that  $\vert\sum_{n\in A_t}x_n-x\vert<\frac{\varepsilon}{2}$ for every $t\in\mathbb{N}$. We have $\sum_{n\in A_t}\vert x_n\vert\geq \vert\sum_{n\in A_t}x_n\vert\geq x-\vert\sum_{n\in A_t}x_n-x\vert>\frac{\varepsilon}{4}$ for each $t\in\mathbb{N}$. Hence $\sum_{n=1}^{\infty} \vert x_{n}\vert\geq\sum_{t=1}^{\infty} \sum_{n\in A_t}\vert x_n\vert\geq \sum_{t=1}^{\infty}\frac{\varepsilon}{4}=\infty$.
 Since $\lim_{n\to\infty}\frac{\vert y_n\vert}{\vert x_n\vert}=\infty$, we have $\sum_{n=1}^{\infty} \vert y_{n}\vert=\infty$. Let $f\in(\mathbb{R}^2)^{*}$ be a convergence functional for a series $\sum_{n=1}^{\infty} (x_{n},y_{n})$, that is $f(x,y)=ax+by$ for some $a,b\in\mathbb{R}$ and $\sum_{n=1}^{\infty} \vert f(x_{n},y_{n})\vert<\infty$. Clearly any of two cases, $a=0,b\neq 0$ or $a\neq 0,b=0$, leads to the contradiction. Suppose that $a\neq 0$ and $b\neq 0$. Let $M$ be a natural number such that $\frac{\vert y_{n}\vert}{\vert x_{n}\vert}>\frac{2a}{b}$ for every $n\geq M$. Thus $\sum_{n=1}^{\infty} \vert ax_{n}+by_{n}\vert\geq\sum_{n=M}^{\infty} \vert ax_{n}+by_{n}\vert\geq a\sum_{n=M}^{\infty} (\frac{b}{a} \vert y_{n}\vert-\vert x_{n}\vert)\geq a\sum_{n=M}^{\infty} \vert x_{n}\vert=\infty$. It means that the series has only one (trivial) convergence functional, so $\on{SR} (x_{n},y_{n})=\mathbb{R}^2$.
\end{proof}

\begin{theorem}\label{reductionplane}
Assume that  $\sum_{n=1}^{\infty} (x_{n},y_{n})$ is a conditionally convergent series, which has the reduction property. Then $L(x_{n},y_{n})=\{(0,1),(0,-1)\}$.
\end{theorem}

\begin{proof}
Suppose that $(\frac{a}{\sqrt{a^2+b^2}},\frac{b}{\sqrt{a^2+b^2}})$ for some $a\neq 0$, $b\in\mathbb{R}$ is a Levy vector of a series $\sum_{n=1}^{\infty} (x_{n},y_{n})$. By Proposition \ref{minimumdwa}(iii) there exists a subseries $\sum_{n=1}^{\infty} (x_{k_n},y_{k_n})$ with $\lim_{n\to\infty}\frac{y_{k_n}}{x_{k_n}}=\frac{b}{a}$. Since $\lim_{n\to\infty}\frac{\vert y_{n}\vert}{\vert x_{n}\vert}=\infty$, we have $\lim_{n\to\infty}\frac{\vert y_{k_n}\vert}{\vert x_{k_n}\vert}=\infty$, which gives us contradiction. By Proposition \ref{minimumdwa}(i) we get $\on{L}(x_{n},y_{n})=\{(0,1),(0,-1)\}$.
\end{proof}

\begin{theorem}\label{reductionplane}
Assume that  $\sum_{n=1}^{\infty} (x_{n},y_{n})$ is a conditionally convergent series, which has the reduction property. Then $\on{A}(x_{n},y_{n})=\mathbb{R}^2$. 
\end{theorem}

\begin{proof}
Fix $(a,b)\in\mathbb{R}^2$ and $\varepsilon>0$. We will define inductively a sequence $(A_i)_{i=0}^{\infty}$ of finite sets of indices $A_{i}=\{k_1^{i}<k_2^{i}<\ldots<k_{m_{i}}^{i}\}$ for $i\in\mathbb{N}_0$ such that for every $p\geq 1$
\begin{itemize}
\item[(i)] $\min A_{p+1}>\max A_{p}$;
\item[(ii)] $\vert\sum_{n\in \bigcup_{i=0}^{2p-1}A_{i}} x_{n}-a\vert<\frac{1}{2^p}$ and $\vert\sum_{n\in \bigcup_{i=0}^{2p-1}A_{i}} y_{n}-b\vert<\frac{4}{2^p}$; 
\item[(iii)] $\vert\sum_{n\in \bigcup_{i=0}^{2p}A_{i}} x_{n}-a\vert<\frac{1}{2^p}$ and $\vert\sum_{n\in \bigcup_{i=0}^{2p}A_{i}} y_{n}-b\vert<\frac{1}{2^p}$ (this case holds also for $p=0$);
\item[(iv)] 
$$\max _{1\leq j\leq m_{2p-1}}\vert\sum_{n\in \bigcup_{i=0}^{2p-2}A_{i}} x_{n}+\sum _{n=1}^{j}x_{k_n^{2p-1}}-a\vert<\frac{4}{2^p}\;\;\text{ and }\;\;\max _{1\leq j\leq m_{2p-1}}\vert\sum_{n\in \bigcup_{i=0}^{2p-2}A_{i}} y_{n}+\sum _{n=1}^{j}y_{k_n^{2p-1}}-b\vert<\frac{4}{2^p};$$ 
\item[(v)] 
$$\max _{1\leq j\leq m_{2p}}\vert\sum_{n\in \bigcup_{i=0}^{2p-1}A_{i}} x_{n}+\sum _{n=1}^{j}x_{k_n^{2p}}-a\vert<\frac{1}{2^p}\;\;\text{ and }\;\;\max _{1\leq j\leq m_{2p}}\vert\sum_{n\in \bigcup_{i=0}^{2p-1}A_{i}} y_{n}+\sum _{n=1}^{j}y_{k_n^{2p}}-b\vert<\frac{4}{2^p}.$$ 
\end{itemize}

By Proposition \ref{plaszczyznasumrange} we know that $\overline{\on{A}(x_n,y_n)}=\mathbb{R}^2$, so one can find a finite set of indices $A_{0}$ such that $\max\{\vert\sum_{n\in A_{0}} x_{n}-a\vert,\vert\sum_{n\in A_{0}}y_{n}-b\vert\}<1$, so we get property (iii) for $p=0$. Assume that for some $p\in\mathbb{N}_0$ we have defined sets $A_0,A_1,\ldots,A_{2p}$ satisfying properties (i)-(v).
By the reduction property applied to $x=a-\sum_{n\in \cup_{i=0}^{2p} A_i} x_{n}$, $N=\max A_{2p}+1$ there exists $A_{2p+1}=\{k_1^{2p+1}<k_2^{2p+1}<\ldots<k_{m_{2p+1}}^{2p+1}\}\subset\{N,N+1,\ldots\}$ such that $\vert\sum _{n\in A_{2p+1}}x_n-x\vert<\frac{1}{2^{p+1}}$ and $\max _{1\leq j\leq m_{2p+1}} \vert \sum _{n=1}^{j}x_{k_n^{2p+1}}\vert<\frac{1}{2^p}$ and $\max _{1\leq j\leq m_{2p+1}} \vert \sum _{n=1}^{j}y_{k_n^{2p+1}}\vert<\frac{1}{2^p}$. We have 
$$\vert a-\sum _{n\in \cup_{i=0}^{2p+1} A_i}x_n\vert= \vert\sum _{n\in A_{2p+1}}x_n-x\vert<\frac{1}{2^{p+1}}$$ 
and 
$$\vert b-\sum _{n\in \cup_{i=0}^{2p+1} A_i}y_n\vert=\vert b-\sum _{n\in \cup_{i=0}^{2p} A_i}y_n\vert+\vert\sum _{n\in A_{2p+1}}y_n\vert<\frac{1}{2^p}+\frac{1}{2^p}=\frac{4}{2^{p+1}}.$$ 
Moreover $$\max _{1\leq j\leq m_{2p+1}}\vert\sum_{n\in \cup_{i=0}^{2p} A_i} x_{n}+\sum _{n=1}^{j}x_{k_n^{2p+1}}-a\vert\leq\max _{1\leq j\leq m_{2p+1}}\vert \sum _{n=1}^{j}x_{k_n^{2p+1}}\vert+\vert\sum_{n\in \cup_{i=0}^{2p} A_i} x_{n}-a\vert<\frac{1}{2^p}+\frac{1}{2^p}=\frac{4}{2^{p+1}}.$$ In simillar way we prove $\max _{1\leq j\leq m_{2p+1}}\vert\sum_{n\in \cup_{i=0}^{2p} A_i} y_{n}+\sum _{n=1}^{j}y_{k_n^{2p+1}}-b\vert<\frac{4}{2^{p+1}}$. We have already checked (ii) and (iv).

 Denote $\delta=\vert a-\sum _{n\in \cup_{i=0}^{2p+1} A_i}x_n\vert<\frac{1}{2^{p+1}}$ and let $M=\min\{n>\max A_{2p+1}: \vert y_{k}\vert>\frac{5}{1-\delta 2^{p+1}}\vert x_{k}\vert \ \text{for each} \ k\geq n \}$. Since $\lim _{n\to\infty}\frac{\vert y_n\vert}{\vert x_n\vert}=\infty$, the inequality $\vert y_{k}\vert>\frac{5}{1-\delta 2^{p+1}}\vert x_{k}\vert$ holds for all but finitely many $k$, so $M$ is well defined. Observe that by Proposition \ref{plaszczyznasumrange} both series  $\sum_{n=1}^{\infty} x_{n}$,   $\sum_{n=1}^{\infty} y_{n}$ are conditionally convergent. Hence  $\sum_{n=M}^{\infty} y_{n}$ is conditionally convergent, so one can find $A_{2p+2}=\{k_1^{2p+2}<k_2^{2p+2}<\ldots<k_{m_{2p+2}}^{2p+2}\}$ with $k_{1}^{2p+2}>M$ such that $\vert\sum_{n\in \cup_{i=0}^{2p+2} A_i} y_{n}-b\vert<\frac{1}{2^{p+1}}$  and $y_n$ for $n\in A_{2p+2}$ have the same sign and 
$$\vert\sum_{n\in \cup_{i=0}^{2p+1} A_i} y_{n}+\sum _{n=1}^{j}y_{k_n^{2p+2}}-b\vert<\vert\sum_{n\in \cup_{i=0}^{2p+1} A_i} y_{n}+\sum _{n=1}^{i}y_{k_n^{2p+2}}-b\vert$$ 
for every $i,j\in\{0,1,\ldots,m_{2p+2}\}$, $i<j$. Hence 
$$\max _{1\leq j\leq m_{2p+2}}\vert\sum_{n\in \cup_{i=0}^{2p+1} A_i} y_{n}+\sum _{n=1}^{j}y_{k_n^{2p+2}}-b\vert\leq \vert\sum_{n\in \cup_{i=0}^{2p+1} A_i} y_{n}-b\vert<\frac{4}{2^{p+1}}$$ 
and 
$$\sum _{n=1}^{m_{2p+2}}\vert y_{k_n^{2p+2}}\vert=\vert\sum _{n=1}^{m_{2p+2}} y_{k_n^{2p+2}}\vert\leq\vert \sum_{n\in \cup_{i=0}^{2p+2} A_i} y_{n}-b\vert+\vert b-\sum_{n\in \cup_{i=0}^{2p+1} A_i} y_{n}\vert<\frac{1}{2^{p+1}}+\frac{4}{2^{p+1}}=\frac{5}{2^{p+1}}.$$
Thus 
$$\sum _{n=1}^{m_{2p+2}}\vert x_{k_n^{2p+2}}\vert<\frac{\frac{1}{2^{p+1}}-\delta}{\frac{5}{2^{p+1}}}\sum _{n=1}^{m_{2p+2}}\vert y_{k_n^{2p+2}}\vert<\frac{1}{2^{p+1}}-\delta$$ 
and consequently 
$$\max _{1\leq j\leq m_{2p+2}}\vert\sum_{n\in \cup_{i=0}^{2p+1} A_i} x_{n}+\sum _{n=1}^{j}x_{k_n^{2p+2}}-a\vert<\frac{1}{2^{p+1}}-\delta+\delta=\frac{1}{2^{p+1}}$$ 
which finishes the proof of properties (iii) and (v). From the construction we also obtain (i).
Conditions (ii), (iii) allows us to construct inductively the sequence $(A_i)_{i=0}^{\infty}$ of finite sets with given properties. Denote $\bigcup_{i=1}^{\infty}A_{i}=\{r_1<r_2<\ldots\}$. By (iv) and (v) we get  $\sum _{n=1}^{\infty}(x_{r_n},y_{r_n})=(a,b)$ and by (i) we have $(a,b)\in \on{A}(x_n,y_n)$. Hence $\on{A}(x_n,y_n)=\mathbb{R}^2$.
\end{proof}

The previous theorem helps us to construct a conditionally convergent series on the plane with two Levy vectors, which achievemet set is $\mathbb{R}^2$. A wide class of such series can be constructed as follows. 

\begin{proposition} \label{dwawektoryiplaszczyznalemat}
Let $(x_n),(y_n)$ be nonincreasing seqeuences of positive numbers tending to $0$ such that $\sum_{n=1}^{\infty} x_{n}=\sum_{n=1}^{\infty} y_{n}=\infty$ and $\lim\limits_{n\to\infty}\frac{\vert y_{n}\vert}{\vert x_{n}\vert}=\infty$.
Let $v_{4n-3}=(-x_{2n-1},-y_{2n-1}), v_{4n-2}=(-x_{2n-1},y_{2n-1}),v_{4n-1}=(x_{2n},y_{2n}), v_{4n}=(x_{2n},-y_{2n})$  for every $n\in\mathbb{N}$. Then $\on{A}(v_n)=\mathbb{R}^2$.
\end{proposition}
\begin{proof}
 Let $\varepsilon>0$, $x\in(0,\varepsilon)$, $N\in\mathbb{N}$. Let $k=\min\{n: y_{n}<\varepsilon\}$ and $p=\max\{k,N\}$. Since $\sum_{n=p}^{\infty} x_{n}=\infty$ one can find  $A=\{k_1<k_2<\ldots<k_m\}$ which is a subset of $(4\mathbb{N}-1 \cup  4\mathbb{N})\cap [p,\infty)$ with the property $4n-1\in A$ if and only if $4n\in A$  for every $n\in\mathbb{N}$ and such that $ 0<x-\sum_{n\in A}x_{n}<\frac{\varepsilon}{2}$.
Recall that $\max_{1\leq j\leq m} \vert \sum_{n=1}^{j}x_{k_n}\vert=\sum_{n\in A}x_{n}<x<\varepsilon$. Moreover for each $r\in\mathbb{N}$ if $2r\leq m$ then $\sum_{n=1}^{2r}y_{k_n}=0$ and if $2r+1\leq m$ then  $\sum_{n=1}^{2r+1}y_{k_n}=y_{2r+1}$. Hence  $\max_{1\leq j\leq m} \vert \sum_{n=1}^{j}y_{k_n}\vert<\varepsilon$.
 If $x\in(-\varepsilon,0)$ then we take $A$ which is a subset of $(4\mathbb{N}-3 \cup  4\mathbb{N}-2)\cap [p,\infty)$  with the property $4n-3\in A$ if and only if $4n-2\in A$  for every $n\in\mathbb{N}$ and such that $ 0<\sum_{n\in A}x_{n}-x<\frac{\varepsilon}{2}$ and get the same results. Hence $\sum_{n=1}^{\infty} v_{n}$ has the reduction property and consequently by Theorem \ref{reductionplane} we get $A(v_n)=\mathbb{R}^2$.
\end{proof}

\begin{example} \label{dwawektoryiplaszczyzna}
\emph{
Let $v_{2n-1}=(x_{2n-1},y_{2n-1})=(\frac{(-1)^{n}}{n},\frac{(-1)^{n}}{\sqrt{n}})$, $v_{2n}=(x_{2n},y_{2n})=(\frac{(-1)^{n}}{n},\frac{(-1)^{n+1}}{\sqrt{n}})$ for every $n\in\mathbb{N}$. Thus $\on{L}(v_n)=\{(0,1),(0,-1)\}$. By Proposition \ref{dwawektoryiplaszczyznalemat} we have $\on{A}(v_{n})=\mathbb{R}^{2}$.}
\end{example}

In the previous section we have given sufficient condition for a conditionally convergent series $\sum_{n=1}^\infty v_n$ to have $\on{A}_{\on{abs}}(v_n)=\R^2$, namely it needs to have three or more Levy vectors. Since $\on{A}_{\on{abs}}(v_n)\subset \on{A}(v_n)$, we also had $\on{A}(v_n)=\R^2$. Now we consider series with two Levy vectors and show that $\on{A}_{\on{abs}}(v_n)$ can be a strict subset of the achievement set.

\begin{proposition}
Let $\sum_{n=1}^\infty y_n$ be a conditionally convergent series and $\sum_{n=1}^\infty x_n$ be an absolutely convergent series of positive terms. Define $v_n=(x_n,y_n)$, then $\on{L}(v_n)=\{(0,1),(0,-1)\}$ and $\on{A}_{\on{abs}}(v_n)\neq \on{A}(v_n)$.
\end{proposition}
\begin{proof}
Obviously $\{(0,1),(0,-1)\}\subset \on{L}(v_n)$. Let $\Vert v\Vert=1$ and $(0,-1)\neq v\neq (0,1)$. Suppose that $v$ is a Levy vector for a series $\sum_{n=1}^\infty v_n$. Hence the series $\sum_{n=1}^\infty \vert x_n\vert$ diverges as a sequence of projections into horizontal line of $(v_n)$, which contradicts with absolute convergence of  $\sum_{n=1}^\infty x_n$. 
\\Let $\sum_{n=1}^\infty v_n=(x,y)$. Then $(x,y)\in \on{A}(v_n)$ as a limit of the series for which it is enough to take $\varepsilon_n=1$ for each $n\in\mathbb{N}$. On the other hand if $x=\sum_{n=1}^\infty\varepsilon_n x_n$ then  $\varepsilon_n=1$ for each $n\in\mathbb{N}$ is a unique representation, because all of terms of $(x_n)$ are positive. Hence $\sum_{n=1}^\infty\varepsilon_n \vert y_n\vert=\sum_{n=1}^\infty \vert y_n\vert=\infty$, so $(x,y)\notin \on{A}_{\on{abs}}(v_n)$.
\end{proof}
Below we give a simply example of a series being a mix of absolutely and conditionally convergent series.
\begin{example}\emph{
Let $v_n=(\frac{1}{2^n},\frac{(-1)^n}{n})$ for $n\in\mathbb{N}$. Then $(1,-\ln2)\in \on{A}(v_n)\setminus\on{A}_{\on{abs}}(v_n)$.}
\end{example}
It shows that it may happen that $\on{A}_{\on{abs}}(v_n)\neq\on{A}(v_n)$. For a given example it easy to see that $\on{A}(v_n)\subset [0,1]\times\mathbb{R}$. If a series $\sum_{n=1}^\infty v_n$ has at least three Levy vectors, then by Theorem \ref{3LevyVectorsAndABS}, $\on{A}(v_n)=\on{A}_{\on{abs}}(v_n)=\R^2$. The next example shows that $\on{A}(v_n)=\R^2$ does not imply that $\on{A}_{\on{abs}}(v_n)=\R^2$ -- clearly by Theorem \ref{3LevyVectorsAndABS} the constructed example has to have exactly two Levy vectors.

\begin{example}\emph{
Let $n_0=0$, $n_k=4^{k^2}+n_{k-1}$ for $k\in\mathbb{N}$. Let $x_{n}=\frac{1}{4^{k^2}}$, $y_{n}=\frac{1}{2^{k}}$  for $n\in (n_{k-1},n_{k}]$. We construct $(v_n)=(v_n^x,v_n^y)$ in the same way as in Proposition \ref{dwawektoryiplaszczyznalemat}. Hence $\on{A}(v_n)=\mathbb{R}^2$. We will show that $\on{A}_{\on{abs}}(v_n)\neq\R^2$, more precisely $\on{A}_{\on{abs}}(v_n)\cap\{\frac{1}{3}\}\times\R=\emptyset$. 
}

\emph{
Suppose to the contrary that there exists $A\subset\mathbb{N}$ such that $\sum_{n\in A} v_n^x=\frac{1}{3}$ and $\sum_{n\in A}\vert v_n^y\vert<\infty$. Then there exists $k\in\mathbb{N}$ such that for each $m\geq k$ set $A$ consists of less then $2^m$ elements $v_n$ for which $\vert v_n^x\vert=\frac{1}{4^{m^2}}$ and $\vert v_n^y\vert=\frac{1}{2^m}$. Hence 
$$\sum_{n\in A, \vert v_n^x\vert\leq\frac{1}{4^{m^2}}} \vert v_n^x\vert\leq\sum_{m=k}^{\infty}\frac{2^m}{4^{m^2}}.$$ 
We will prove inductively that $\sum_{m=n+1}^{\infty}\frac{2^m}{4^{m^2}}<\frac{1}{8\cdot 4^{n^2}}$ for $n\in\mathbb{N}$. Note that $\sum_{m=1}^{\infty}\frac{2^m}{4^{m^2}}<\sum_{m=1}^{\infty}\frac{2^m}{4^{m}}=1$. Assume that $\sum_{m=n}^{\infty}\frac{2^m}{4^{m^2}}<\frac{1}{4^{(n-1)^2}}$ for some $n\in\mathbb{N}$. We will show that $\sum_{m=n+1}^{\infty}\frac{2^m}{4^{m^2}}<\frac{1}{8\cdot 4^{n^2}}<\frac{1}{4^{n^2}}$. We have
$$\sum\limits_{m=n+1}^{\infty}\frac{2^m}{4^{m^2}}=\sum\limits_{m=n}^{\infty}\frac{2^{m+1}}{4^{(m+1)^2}}=\sum\limits_{m=n}^{\infty}\frac{2^{m+1}}{4^{m^2}\cdot 4^{(m+1)^2-m^2}}<\frac{2} {4^{(n+1)^2-n^2}}\sum\limits_{m=n}^{\infty}\frac{2^{m}}{4^{m^2}}.$$
Thus by the inductive assumption we obtain
$$\sum\limits_{m=n+1}^{\infty}\frac{2^m}{4^{m^2}}<\frac{2} {4^{(n+1)^2-n^2+(n-1)^2}}=\frac{2} {4^{n^2+2}}=\frac{1}{8\cdot 4^{n^2}}.$$
In particular we have 
$$\sum\limits_{n\in A, \vert v_n^x\vert\leq\frac{1}{4^{(k+1)^2}}} \vert v_n^x\vert<\frac{1}{8\cdot 4^{k^2}}.$$ 
Since $\sum_{n\in A} v_n^x=\frac{1}{3}$, then $\vert \sum_{n\in A, \vert v_n^x\vert\geq\frac{1}{4^{k^2}}}v_n^x -\frac{1}{3} \vert<\frac{1}{8\cdot 4^{k^2}}$. Note that $\sum_{n\in A, \vert v_n^x\vert\geq\frac{1}{4^{k^2}}}v_n^x=\frac{p}{4^{k^2}}$ for some $p\in \mathbb{Z}$. We have 
$$\sum\limits_{n=1}^{k^2}\frac{4^{n^2-n}}{4^{n^2}}=\sum\limits_{n=1}^{k^2}\frac{1}{4^{n}}=\frac{1-\frac{1}{4^{k^2}}}{3}=\frac{1}{3}-\frac{1}{3\cdot 4^{k^2}}<\frac{1}{3}<\frac{1}{3}+\frac{2}{3\cdot 4^{k^2}}.$$
Thus $\min\{\vert \frac{a}{4^{k^2}}-\frac{1}{3}\vert : a\in\mathbb{Z}\}=\frac{1}{3\cdot 4^{m^2}}$. Hence $\frac{1}{3\cdot 4^{k^2}}\leq\vert\frac{p}{4^{k^2}}-\frac{1}{3}\vert<\frac{1}{8\cdot 4^{k^2}}$ which yields a contradiction.}
\end{example}

\section{Extreme example and open problems}\label{SectionExample}

As we have mentioned in the Introduction this section is devoted to the construction of a series in the plane such each  vertical section of its achievement set has at most one element.

We define positive real numbers $x_1\geq x_2\geq\dots$, natural numbers $0=N_0<N_1<N_2,\dots$ inductively as follows. 
First we define $x_1=1$ and $N_1=1$. Define $\delta_1=1$. Let $n$ be such that $\frac{2^n\delta_1}{4^{1+3}}=1$. Put $N_{2}=N_1+2^n$ and $x_{N_1+i}=\frac{\delta_1}{2\cdot 4^{2}}+\frac{1}{2^{2^n+i}}$ for $i=1,\dots,2^n$. Put $\delta_2:=\min\{\vert\sum_{i=N_{1}+1}^{N_{2}}\xi_ix_i\vert:\xi_i\in\{-1,0,1\}$ and $\xi_i\neq 0$ for some $i\}=\frac{1}{2^{2\cdot 2^n}}=\frac{1}{4^{2^n}}$. Suppose that we have already defined $N_1<N_2<\dots<N_k$ and $x_i$ for $i\leq N_k$ such that (1)--(3) hold. Let $n$ be such that $\frac{2^n\delta_k}{4^{k+3}}=1$. Put $N_{k+1}=N_k+2^n$ and $x_{N_k+i}=\frac{\delta_k}{2\cdot 4^{k+1}}+\frac{1}{2^{2^n+i}}$ for $i=1,\dots,2^n$. Put $\delta_{k+1}:=\min\{\vert\sum_{i=N_{k}+1}^{N_{k+1}}\xi_ix_i\vert:\xi_i\in\{-1,0,1\}$ and $\xi_i\neq 0$ for some $i\}=\frac{1}{2^{2\cdot 2^n}}=\frac{1}{4^{2^n}}$. Note that $\delta_{k+1}<\frac{\delta_k}{\cdot 4^{k+2}}$, $\sum_{i=N_{k-1}+1}^{N_{k}}x_i\geq 1$ and  $x_i<\frac{\delta_k}{4^{k+1}}$ for $i=N_{k}+1,\dots,N_{k+1}$.

\begin{lemma}\label{KeyLemmaExample}
Let $(\varepsilon_i),(\varepsilon_i')\in\{0,1\}^\N$ be distinct and such that $\sum_{i=1}^\infty(-1)^i\varepsilon_ix_i=\sum_{i=1}^\infty(-1)^i\varepsilon_i'x_i$. Then
\[
\Big\vert\sum_{i=N_k+1}^{N_{k+1}}(-1)^i\varepsilon_ix_i-\sum_{i=N_k+1}^{N_{k+1}}(-1)^i\varepsilon_i'x_i\Big\vert=\Big\vert\sum_{i\in(N_k,N_{k+1}]\cap 2\N, \varepsilon_i\neq\varepsilon_i'}x_i-\sum_{i\in(N_k,N_{k+1}]\cap (2\N-1), \varepsilon_i\neq\varepsilon_i'}x_i\Big\vert\geq\frac{\delta_k}{2}
\]
for infinitely many $k$. Moreover, for infinitely many $k$ 
\[
\Big\vert\vert\{i\in(N_k,N_{k+1}]\cap 2\N: \varepsilon_i\neq\varepsilon_i'\}\vert-\vert\{i\in(N_k,N_{k+1}]\cap (2\N-1): \varepsilon_i\neq\varepsilon_i'\}\vert\Big\vert>4^k.
\]
\end{lemma}

\begin{proof} We start from proving the following.\\
\textbf{Claim. }\textit{Let $j,k\in\N$ be such that $\varepsilon_j\neq\varepsilon_j'$ and $N_k<j\leq N_{k+1}$. Then 
\[
\Big\vert\sum_{i=N_k+1}^{N_{k+1}}(-1)^i\varepsilon_ix_i-\sum_{i=N_k+1}^{N_{k+1}}(-1)^i\varepsilon_i'x_i\Big\vert\geq\delta_{k+1}.
\]
}

We have 
\[
\sum_{i\in(N_k,N_{k+1}]\cap 2\N}\varepsilon_ix_i=u\frac{\delta_k}{2\cdot 4^{k+1}}+\sum_{i\in(N_k,N_{k+1}]\cap 2\N}\varepsilon_i2^{-2^{N_{k+1}-N_k}-i}
\]
where $u$ stands for $\big\vert\{i\in(N_k,N_{k+1}]\cap 2\N:\varepsilon_i=1\}\big\vert$. Similarly 
\[
\sum_{i\in(N_k,N_{k+1}]\cap 2\N}\varepsilon_i'x_i=u'\frac{\delta_k}{2\cdot 4^{k+1}}+\sum_{i\in(N_k,N_{k+1}]\cap 2\N}\varepsilon_i'2^{-2^{N_{k+1}-N_k}-i}
\]
\[
\sum_{i\in(N_k,N_{k+1}]\cap 2\N-1}\varepsilon_ix_i=v\frac{\delta_k}{2\cdot 4^{k+1}}+\sum_{i\in(N_k,N_{k+1}]\cap 2\N-1}\varepsilon_i2^{-2^{N_{k+1}-N_k}-i}
\]
\[
\sum_{i\in(N_k,N_{k+1}]\cap 2\N-1}\varepsilon_i'x_i=v'\frac{\delta_k}{2\cdot 4^{k+1}}+\sum_{i\in(N_k,N_{k+1}]\cap 2\N-1}\varepsilon_i'2^{-2^{N_{k+1}-N_k}-i}
\]
where $u=\big\vert\{i\in(N_k,N_{k+1}]\cap 2\N:\varepsilon_i'=1\}\big\vert$, $w=\big\vert\{i\in(N_k,N_{k+1}]\cap 2\N-1:\varepsilon_i=1\}\big\vert$ and $w'=\big\vert\{i\in(N_k,N_{k+1}]\cap 2\N-1:\varepsilon'_i=1\}\big\vert$. Let $\xi_i=(-1)^i(\varepsilon_i-\varepsilon_i')\in\{-1,0,1\}$. Note that $\xi_{j}\neq 0$. Thus  
\[
\Big\vert\sum_{i=N_k+1}^{N_{k+1}}(-1)^i\varepsilon_ix_i-\sum_{i=N_k+1}^{N_{k+1}}(-1)^i\varepsilon_i'x_i\Big\vert=\Big\vert 
\frac{\vert u-u'\vert\delta_k}{2\cdot 4^{k+1}}-\frac{\vert v-v'\vert\delta_k}{2\cdot 4^{k+1}}+\sum_{i=N_k+1}^{N_{k+1}}\xi_i2^{-2^{N_{k+1}-N_k}-i}\Big\vert=\Big\vert t\frac{\delta_k}{2\cdot 4^{k+1}}+R_k\Big\vert
\]
where $t$ is an integer and $0<\delta_{k+1}=\frac{1}{4^{N_{k+1}-N_k}}\leq R_k<\frac{1}{2^{N_{k+1}-N_k}}$. If $t=0$, then
\[
\Big\vert t\frac{\delta_k}{2\cdot 4^{k+1}}+R_k\Big\vert=\vert R_k\vert\geq\delta_{k+1}.
\]
If $t\neq 0$, then
\[
\Big\vert t\frac{\delta_k}{2\cdot 4^{k+1}}+R_k\Big\vert\geq\Big\vert \frac{\delta_k}{2\cdot 4^{k+1}}-\frac{1}{2^{N_{k+1}-N_k}}\Big\vert\geq\delta_{k+1}.
\]
This finishes the proof of the Claim. \vspace{0.5 cm}

Suppose on the contrary that for all but finitely many $k$ we have
\begin{equation}\label{eq1}
\Big\vert\sum_{i=N_l+1}^{N_{l+1}}(-1)^i\varepsilon_ix_i-\sum_{i=N_l+1}^{N_{l+1}}(-1)^i\varepsilon_i'x_i\Big\vert<\frac{\delta_l}{2}.
\end{equation}
Note that there is $k$ such that
\begin{equation}\label{eq2}
\Big\vert\sum_{i=N_l+1}^{N_{l+1}}(-1)^i\varepsilon_ix_i-\sum_{i=N_l+1}^{N_{l+1}}(-1)^i\varepsilon_i'x_i\vert\geq\delta_{k+1} \text{ and }\vert\sum_{i=N_l+1}^{N_{l+1}}(-1)^i\varepsilon_ix_i-\sum_{i=N_l+1}^{N_{l+1}}(-1)^i\varepsilon_i'x_i\Big\vert<\frac{\delta_l}{2}
\end{equation}
for all $l\geq k$. Indeed, if \eqref{eq1} holds for every $l$, then by the assumption there is $j$ with $\varepsilon_j\neq\varepsilon_j'$. Let $k$ be such that $N_k<j\leq N_{k+1}$. Then by the Claim we obtain \eqref{eq2}. If \eqref{eq1} holds for all but finitely many $l$, then find $k$ such that \eqref{eq1} holds for every $l\geq k$ and 
\[
\Big\vert\sum_{i=N_k+1}^{N_{k+1}}(-1)^i\varepsilon_ix_i-\sum_{i=N_k+1}^{N_{k+1}}(-1)^i\varepsilon_i'x_i\Big\vert\geq\frac{\delta_k}{2}>\delta_{k+1}.
\]
Then
\[
\delta_{k+1}\leq \Big\vert\sum_{i=1}^{N_{k+1}}(-1)^i\varepsilon_ix_i-\sum_{i=1}^{N_{k+1}}(-1)^i\varepsilon_i'x_i\Big\vert\leq
\Big\vert \sum_{i=1}^\infty(-1)^i\varepsilon_ix_i-\sum_{i=1}^\infty(-1)^i\varepsilon_i'x_i\Big\vert +
\]
\[
\sum_{l=k+1}^\infty\Big\vert\sum_{i=N_l+1}^{N_{l+1}}(-1)^i\varepsilon_ix_i-\sum_{i=N_l+1}^{N_{l+1}}(-1)^i\varepsilon_i'x_i\Big\vert\leq \frac{\delta_{k+1}}{2}+\frac{\delta_{k+2}}{2}+\frac{\delta_{k+3}}{2}+\dots<\frac{3\delta_{k+1}}{4}.
\]
A contradiction.

The moreover part of the assertion follows from the fact that $x_i<\frac{\delta_k}{4^{k+1}}$ for $i\in(N_k,N_{k+1}]$.
\end{proof}

Let $y_i=\frac{1}{2^k}$ for $i\in(N_k,N_{k+1}]$. Consider a series $\sum_{i=1}^\infty((-1)^ix_i,(-1)^iy_i)$. 
\begin{lemma}\label{SecondLemmaExample}
Let $(\varepsilon_i)\in\{0,1\}^\N$ be such that $\sum_{i=1}^\infty\varepsilon_i((-1)^ix_i,(-1)^iy_i)$ is convergent. Then for every $(\varepsilon_i')\in\{0,1\}^\N$ such that $(\varepsilon_i')\neq (\varepsilon_i)$ and $\sum_{i=1}^\infty\varepsilon_i(-1)^ix_i=\sum_{i=1}^\infty\varepsilon_i'(-1)^ix_i$ a series $\sum_{i=1}^\infty\varepsilon_i'((-1)^ix_i,(-1)^iy_i)$ is divergent.
\end{lemma}

\begin{proof}
By the moreover part of Lemma \ref{KeyLemmaExample} there are infinitely many $k$ such that 
\[
\Big\vert\vert\{i\in(N_k,N_{k+1}]\cap 2\N: \varepsilon_i\neq\varepsilon_i'\}\vert-\vert\{i\in(N_k,N_{k+1}]\cap (2\N-1): \varepsilon_i\neq\varepsilon_i'\}\vert\Big\vert>4^k.
\]
That means that for infinitely many $k$ 
\begin{equation}
\Big\vert\sum_{i=N_k+1}^{N_{k+1}}\varepsilon_i(-1)^iy_i-\sum_{i=N_k+1}^{N_{k+1}}\varepsilon_i'(-1)^iy_i\Big\vert=\Big\vert\sum_{i=N_k+1}^{N_{k+1}}\varepsilon_i(-1)^i\frac{1}{2^k}-\sum_{i=N_k+1}^{N_{k+1}}\varepsilon_i'(-1)^i\frac{1}{2^k}\Big\vert\geq \frac{4^k}{2^k}=2^k.
\end{equation}
By the Cauchy condition for $\sum_{i=1}^\infty\varepsilon_i((-1)^ix_i,(-1)^iy_i)$ for almost all $k$ 
\[
\Big\vert\sum_{i=N_k+1}^{N_{k+1}}\varepsilon_i(-1)^iy_i\Big\vert\leq 1.
\]
Therefore for infinitely many $k$
\[
\Big\vert\sum_{i=N_k+1}^{N_{k+1}}\varepsilon_i(-1)^iy_i\Big\vert\geq 2^k-1.
\]
This shows that the series $\sum_{i=1}^\infty\varepsilon_i'((-1)^ix_i,(-1)^iy_i)$ does not fulfill the Cauchy condition. 
\end{proof}

By Lemma \ref{SecondLemmaExample} the series $\sum_{i=1}^\infty((-1)^ix_i,(-1)^iy_i)$ has the desired property that each vertical section of $\on{A}((-1)^ix_i,(-1)^iy_i)$ has at most one element.  

We would like to end the paper with the list of open questions. 
\begin{problem}\label{Q1}
Let $\sum_{n=1}^\infty(x_n,y_n)$ be a conditionally convergent series on the plane with $\on{SR}(x_n,y_n)=\R^2$. Is the reduction property necessary for $\on{A}(x_n,y_n)=\R^2$?
\end{problem}
It is sometimes hard to check the reduction property. In particular we do not know the answer for the following.
\begin{problem}\label{Q2}
Does $\on{A}(\frac{(-1)^n}{n},\frac{(-1)^n}{\sqrt{n}})=\R^2$?
\end{problem}
We have constructed two series with $\on{SR}(x_n)=\mathbb{R}^2$ and $\on{A}(x_n)\neq\mathbb{R}^2$ -- the first in \cite{BGM} and the second in this section. Both of them are small subset of the plane. We would like to know if this is the general phenomena. 
\begin{problem}
Let $\sum_{n=1}^{\infty} x_{n}$ be a conditionally convergent series on the plane with $\on{SR}(x_n)=\mathbb{R}^2$. Is it true that if $\on{A}(x_n)\neq\mathbb{R}^2$, then $\on{A}(x_n)$ is a set of Lebesgue measure zero or of the first Baire category?
\end{problem}

Suppose that $\sum_{n=1}^\infty x_n$ is a series with $\on{SR}(x_n)=\mathbb{R}^2$ and $\on{A}(x_n)\neq\mathbb{R}^2$. Take $a\in \R^2\setminus \on{A}(x_n)$. Since $\on{SR}(x_n)=\mathbb{R}^2$, there is a permutation $\sigma$ of natural numbers such that $\sum_{n=1}^\infty x_{\sigma(n)}=a$. Therefore $a\in\on{A}(x_{\sigma(n)})$. This shows that the achievement set of conditionally convergent series is not invariant under taking rearrangements. On the other hand a rearrangement does not affect Levy vectors. Therefore if $\sum_{n=1}^\infty x_n$ has at least three Levy vectors, then $\on{A}(x_{\sigma(n)})=\R^2$ for any permutation $\sigma$. We do not know what happens if the series has only two Levy vectors.

\begin{problem}
Let $\sum_{n=1}^{\infty} x_{n}$ be a conditionally convergent series on the plane such that $\on{SR}(x_n)=\mathbb{R}^2$, $\on{A}(x_n)=\mathbb{R}^2$ and $\vert\on{L}(x_n)\vert=2$. Is it true that $\on{A}(x_{\sigma(n)})=\mathbb{R}^2$ for any permutation $\sigma$?
\end{problem}

\end{document}